\documentclass[reqno,oneside]{amsart}
\usepackage{hyperref}
\usepackage{geometry}
\usepackage[ansinew]{inputenc}
\usepackage{graphicx}
\usepackage{amsmath}
\usepackage{amsthm}
\usepackage{amssymb, color}%
\usepackage[numbers, square]{natbib}
\usepackage{mathrsfs}
\usepackage{bbm}
\usepackage{tikz}
\usepackage{mathptmx}

\newcommand{\R}{\mathbb{R}}
\newcommand{\E}{\mathbb{E}}
\newcommand{\N}{\mathbb{N}}

\renewcommand{\P}{\mathbb{P}}

\newcommand{\1}{\mathbbm{1}}
\newcommand{\todistr}{\overset{d}{\underset{n\to\infty}\longrightarrow}}
\newcommand{\todistrt}{\overset{d}{\underset{t\to\infty}\longrightarrow}}

\newcommand{\toprobab}{\overset{P}{\underset{n\to\infty}\longrightarrow}}
\newcommand{\toprobabt}{\overset{P}{\underset{t\to\infty}\longrightarrow}}

\newcommand{\toas}{\overset{a.s.}{\underset{n\to\infty}\longrightarrow}}

\newcommand{\toast}{\overset{a.s.}{\underset{t\to\infty}\longrightarrow}}

\theoremstyle{plain}
\newtheorem{theorem}{Theorem}[section]
\newtheorem{lemma}[theorem]{Lemma}
\newtheorem{corollary}[theorem]{Corollary}
\newtheorem{proposition}[theorem]{Proposition}
\theoremstyle{definition}

\newtheorem{example}[theorem]{Example}
\theoremstyle{remark}
\newtheorem{remark}[theorem]{Remark}

\begin{document}

\title[Moderate parts in regenerative compositions: the case of regular variation]{Moderate parts in regenerative compositions:\\ the case of regular variation}

\author{Dariusz Buraczewski}
\address{Dariusz Buraczewski, Mathematical Institute University of Wroclaw, Pl. Grunwaldzki 2/4
50-384 Wroclaw, Poland}\email{dbura@math.uni.wroc.pl}

\author{Bohdan Dovgay}
\address{Bohdan Dovgay, Faculty of Computer Science and Cybernetics, Taras Shev\-chen\-ko National University of Kyiv, 01601 Kyiv, Ukraine}
\email{bogdov@gmail.com}

\author{Alexander Marynych}
\address{Alexander Marynych, Faculty of Computer Science and Cybernetics, Taras Shev\-chen\-ko National University of Kyiv, 01601 Kyiv, Ukraine}
\email{marynych@unicyb.kiev.ua}

\begin{abstract}
A regenerative random composition of integer $n$ is constructed by allocating $n$ standard exponential points over a countable number of intervals, comprising the complement of the closed range of a subordinator $S$. Assuming that the L\'{e}vy measure of $S$
is infinite and regularly varying at zero of index $-\alpha$, $\alpha\in(0,\,1)$, we find an explicit threshold  $r=r(n)$, such that the number $K_{n,\,r(n)}$ of blocks of size $r(n)$ converges in distribution without any normalization to a mixed Poisson distribution. The sequence $(r(n))$ turns out to be regularly varying with index $\alpha/(\alpha+1)$ and the mixing distribution is that of the exponential functional of $S$. \textcolor{black}{The result is derived as a consequence of a general Poisson limit theorem for an infinite occupancy scheme with power-like decay of the frequencies.} We also discuss asymptotic behavior of $K_{n,\,w(n)}$ in cases when $w(n)$ diverges but grows slower than $r(n)$. Our findings complement previously known strong laws of large numbers for $K_{n,\,r}$ in case of a fixed $r\in\N$. As a key tool we employ new Abelian theorems for Laplace--Stiletjes transforms of regularly varying functions with the indexes of regular variation diverging to infinity. 
\end{abstract}



\maketitle

\section{Introduction and main results}
Let $S=(S(t))_{t\geq 0}$ be a drift-free subordinator with no killing and a L\'{e}vy measure $\nu$ on $(0,\infty)$. The classical It\^{o} decomposition reads as
$$
S(t)=\sum_{k:\tau_k\leq t}j_k,\quad t\geq 0,
$$
with probability one, where $\sum\delta_{(\tau_k,j_k)}$ is a Poisson point process with values in $[0,\,\infty)\times (0,\, \infty)$ and intensity measure ${\bf LEB}\times \nu$. Here $\delta_x$ is a Dirac point measure at $x$ and ${\bf LEB}$ is the standard Lebesgue measure on $[0,\,\infty)$.  A random closed subset of $[0,\,\infty)$ defined by $\mathcal{R}:={\rm cl}{\{S(t):t\geq 0\}}$ is called the range of $S$ and the open complement $\mathcal{R}^c:=[0,\,\infty)\setminus \mathcal{R}$ can be expressed as the union $\mathcal{R}^c=\cup_{k}(S(\tau_k-),\,S(\tau_k))=\cup_{k}(S(\tau_k-),\,S(\tau_k-)+j_k)$ of countably many open intervals, where $\{\tau_k\}$ is the set of jump epochs of $S$. We call the disjoint intervals comprising $\mathcal{R}^c$ {\it boxes}. Further, let $(E_k)_{k\in\N}$ be a sequence of independent copies of a random variable $E$ with the standard exponential distribution, and $(E_k)_{k\in\N}$ is independent of $S$. The points $(E_k)_{k\in\N}$ are called {\it balls} and are dropped one by one on $[0,\,\infty)$. Since the Lebesgue measure of $\mathcal{R}$ is zero with probability one, see Proposition 1.8 in \cite{Bertoin:1997}, each ball $E_j$ with probability one falls into one of the boxes $(S(\tau_k-),\,S(\tau_k))$. \textcolor{black}{Restricting attention to the first $n$ balls, we let $\mathcal{C}_n$ denote the vector of nonzero occupancy numbers of the intervals $(S(\tau_k-),\,S(\tau_k))$, written in their natural order, that is, the first coordinate of $\mathcal{C}_n$ is the number of balls in the first occupied interval of $\mathcal{R}^c$ counting from left to right, the second component of $\mathcal{C}_n$ is the number of balls in the next occupied interval and so on. When $n$ varies, a family $(\mathcal{C}_n)_{n\in\N}$ defines a coherent sequence of random compositions in the following sense. For every $n\in\N$, the sum of coordinates of $\mathcal{C}_n$ is equal to $n$ and one can pass from the composition of $n$ given by $\mathcal{C}_{n}$ to the composition of $n-1$ defined by $\mathcal{C}_{n-1}$ be removing the point $E_n$ from a box it occupies.}

The family $(\mathcal{C}_n)_{n\in\N}$ possesses a distinguishing property called {\it regeneration} inherited from the regenerative property of the set $\mathcal{R}$ combined with the memoryless property of the exponential distribution. Consider composition $\mathcal{C}_n$ of integer $n$ and suppose that the first summand is equal to $m<n$, then deleting this part yields a composition on $n-m$ which has the same distribution as $\mathcal{C}_{n-m}$. In view of this property, the sequence
$(\mathcal{C}_n)_{n\in\N}$ is called {\it regenerative composition structure}, as introduced in \cite{Gnedin+Pitman:2005}.

For $n,k,r\in\N$, set
$$
\mathcal{Z}_{n,\,k}:=\#\{1\leq j\leq n:E_j\in (S(\tau_k-),\,S(\tau_k))\}\quad\text{and}\quad K_{n,\,r}:=\sum_{k\geq 1}\1_{\{\mathcal{Z}_{n,\,k}=r\}}.
$$
Thus, $K_{n,\,r}$ is the number of boxes occupied by exactly $r$ balls and $K_n:=\sum_{r=1}^{n}K_{n,\,r}$ is the total number of occupied boxes. The asymptotic analysis, as $n\to\infty$, of $K_n$ and $K_{n,\,r}$ for regenerative compositions has received a considerable attention in the past decades. The model exhibits a wide range of possible asymptotic regimes depending on the tail behavior of the governing L\'{e}vy measure $\nu$. It is common to distinguish three situations:
\begin{itemize}
\item[(i)] the case of finite $\nu$ in which the corresponding construction is called {\it the Bernoulli sieve};
\item[(ii)] the case where $\nu$ is infinite and the function $y\mapsto \nu([y,\,\infty))$ is slowly varying at zero;
\item[(iii)] the case where $\nu$ is infinite and the function $y\mapsto \nu([y,\,\infty))$ is regularly varying at zero with index $-\alpha$, $\alpha\in (0,\,1)$.
\end{itemize}
In case (i), in which the subordinator $S$ is a compound Poisson process, further subdivision stems from the tail behavior (both at zero and infinity) of the distribution of the generic jump of $S$. For example, if this distribution has finite mean and is nonlattice, then there exists a nondegenerate random vector $(K_1,K_2,\ldots)$ such that
\begin{equation}\label{eq:k_n_r_bs}
(K_{n,\,1},K_{n,\,2},\ldots)\todistr (K_1,K_2,\ldots),
\end{equation}
see Theorem 3.3 in \cite{Gnedin+Iksanov+Roesler:2008}. Under the additional assumption that the distribution of the generic jump of $S$ belongs to a domain of attraction of a stable law with the index of stability lying in $(1,2]$, the total number of occupied boxes $K_n$, properly centred and normalized, converges in distribution to a stable law, see \cite{Gnedin:2004} and \cite{Gnedin+Iksanov+Marynych:2010}. Furthermore, a functional limit theorem for the process $[0,\,1]\ni z\mapsto \sum_{r\leq n^z}K_{n,\,r}$, is also available, see Theorem 2.2 in \cite{Alsmeyer+Iksanov+Marynych:2017}. In particular, the aforementioned results demonstrate that the main contribution to $K_n$ in case (i) is given by $K_{n,\,r}$'s with $r$ lying between $n^{a}$ and $n^{b}$, $0<a<b\leq 1$, whereas small counts $K_{n,\,r}$, with $r$ fixed, are negligible. A completely different picture occurs in case (iii) in which $K_n$ has the same magnitude as $K_{n,\,r}$, $r=1,2,\ldots$ in the following sense. \textcolor{black}{If $\nu([y,\,\infty))~\sim~y^{-\alpha}\ell(1/y)$, as $y\downarrow 0$, where $\ell$ is a slowly varying at infinity function, then
\begin{equation}\label{eq:k_n_r_reg_var}
\frac{K_n}{\Gamma(1-\alpha)n^{\alpha}\ell(n)}\toas \mathcal{I}_{\alpha}\quad\text{and}\quad \frac{K_{n,\,r}}{\Gamma(1-\alpha)n^{\alpha}\ell(n)}\toas (-1)^{r-1}\binom{\alpha}{r} \mathcal{I}_{\alpha},\quad r\in\N,
\end{equation}}
where $\mathcal{I}_{\alpha}$ is the exponential functional of the subordinator $S$, that is,
\begin{equation}\label{eq:exp_func_def}
\mathcal{I}_{\alpha}:=\int_0^{\infty}e^{-\alpha S(\tau)}{\rm d}{\tau},
\end{equation}
see Theorem 4.1 in \cite{Gnedin+Pitman+Yor:2006}. Intermediate regimes in case (ii), in which $K_{n,\,r}$ diverges but grows slower than $K_n$, see \cite{Gnedin+Barbour:2006}, \cite{Gnedin+Iksanov:2012} and \cite{Gnedin+Pitman+Yor:2006-2}, make the picture even more diverse.

The main purpose of this paper is a further investigation of regenerative compositions in case (iii). A natural question arising while comparing \eqref{eq:k_n_r_bs} and \eqref{eq:k_n_r_reg_var} is the following. What is a threshold $r=r(n)$ such that $K_{n,\,r(n)}$ converges without any normalization, as $n\to\infty$, to a finite nondegenerate limit and what is the limit? \textcolor{black}{An intuitive answer to the first part, which can be guessed from \eqref{eq:k_n_r_reg_var} is: pick $r=r(n)$ such that
\begin{equation}\label{eq:r_naive}
\lim_{n\to\infty}\Gamma(1-\alpha)n^{\alpha}\ell(n)(-1)^{r(n)-1}\binom{\alpha}{r(n)}=\lim_{n\to\infty}\frac{\alpha \Gamma(r(n)-\alpha)n^{\alpha}\ell(n)}{\Gamma(1+r(n))}=\lim_{n\to\infty}\frac{\alpha n^{\alpha}\ell(n)}{(r(n))^{1+\alpha}}=1.
\end{equation}
This answer turns out to be {\it almost} correct as out first main result demonstrates.} Throughout the paper the notation $f(t)~\sim~g(t)$ is used to denote asymptotic equivalence, that is, $\lim_{t\to\infty}f(t)/g(t)=1$.

\begin{theorem}\label{thm:main}
Assume that $S$ is a drift-free subordinator with no killing and a L\'{e}vy measure $\nu$ on $(0,\infty)$ satisfying
\begin{equation}\label{eq:reg_var_nu}
\nu([y,\,\infty))~\sim~y^{-\alpha}\ell(1/y),\quad y\downarrow 0,
\end{equation}
for some $\alpha \in (0,1)$ and a slowly varying at infinity function $\ell$. Let $r=r(t)$ by any positive function such that
\begin{equation}\label{eq:r_def}
\lim_{t\to\infty}\frac{\alpha t^{\alpha}}{(r(t))^{\alpha+1}}\ell\left(\frac{t}{r(t)}\right)=1,
\end{equation}
and $r_i=r_i(t)$, $i=1,\ldots,m$ be integer-valued functions such that $r_i(t)~\sim~u_i r(t)$ for some $0<u_1<u_2<\cdots<u_m<\infty$, as $t\to\infty$. Then
$$
\left(K_{n,\,r_1(n)},K_{n,\,r_2(n)},\ldots,K_{n,\,r_m(n)}\right) \todistr (\mathcal{P}_1,\mathcal{P}_2,\ldots,\mathcal{P}_m),
$$
where, given the subordinator $S$, $(\mathcal{P}_i)_{i=1,\ldots,m}$ are mutually independent Poisson random variables with $\E \left[\mathcal{P}_i |S\right] = u_i^{-\alpha-1}\mathcal{I}_{\alpha}$ and $\mathcal{I}_{\alpha}$ is defined by \eqref{eq:exp_func_def}.
\end{theorem}

\begin{remark}\label{rem:rer_var_r}
The existence and uniqueness (up to asymptotic equivalence) of a function $r$ satisfying \eqref{eq:r_def} follows by a standard argument involving de Bruijn conjugates. Put $\ell_1(t):=\alpha^{1/(\alpha+1)}\ell^{1/(\alpha+1)}(t)$. Then $\ell_1$ is slowly varying at infinity and thus possesses a de Bruijn conjugate, say $\ell_1^{\#}$, see Theorem 1.5.13 in \cite{BGT}, such that $\lim_{t\to\infty}\ell_1^{\#}(t)\ell_1(t\ell_1^{\#}(t))=1$. Then $r(t)=t^{\alpha/(\alpha+1)}/\ell_1^{\#}(t^{1/(\alpha+1)})$ satisfies \eqref{eq:r_def}. In particular, $r$ is regularly varying at infinity with index $\alpha/(\alpha+1)\in (0,\,1/2)$. \textcolor{black}{Note that $r$ defined by the last equality in \eqref{eq:r_naive}, and $r$ defined by \eqref{eq:r_def}, always have the same index of regular variation $\alpha/(\alpha+1)$ but might have different slowly varying factors. This explains the usage of the word 'almost' after formula \eqref{eq:r_naive}. Is some cases, say when $\ell$ is a constant, both definitions are identical.}
\end{remark}

\begin{remark}
\textcolor{black}{We shall put Theorem \ref{thm:main} in a more general context of infinite occupancy schemes in Section \ref{sec:karlin}. In fact, Theorem \ref{thm:main} will be derived from a general Poisson limit theorem for infinite occupancy schemes, a result that might be of independent interest.}
\end{remark}

In order to formulate our next results, let us introduce the following classes of functions:
\begin{itemize}
\item $\mathcal{W}^{r}$, a class of positive functions defined in a neighborhood of infinity such that $w\in\mathcal{W}^r$ iff $w(t)=o(r(t))$ and $w(t)\to\infty$ as $t\to\infty$, where $r$ is defined by \eqref{eq:r_def};
\item $\mathcal{W}^{r}_{RV}$, a subclass of $\mathcal{W}^{r}$ comprised of functions which are regularly varying at infinity, necessarily with index of regular variation in $[0,\,\alpha/(\alpha+1)]$;
\item $\mathcal{Q}$, a class of positive functions defined in a neighborhood of infinity such that $q\in\mathcal{Q}$ iff $q(t)=o(t)$ and $q(t)\to\infty$ as $t\to\infty$.
\end{itemize}
Obviously, $\mathcal{W}^{r}_{RV}\subset \mathcal{W}^{r}\subset \mathcal{Q}$. We also denote by $\mathcal{N}$ the class of $\mathbb{N}$-valued functions.

An interesting question naturally occurring upon examination of Theorem \ref{thm:main} is the following. What happens with $K_{n,\,w(n)}$ for the functions $w\in\mathcal{W}^{r}\cap\mathcal{N}$, that is, for $w(n)$ growing slower than $r(n)$? The answer to this question under an additional assumption of regular variation is provided by our second main result.

\begin{theorem}\label{thm:main2}
Let $w\in\mathcal{W}^{r}_{RV}\cap\mathcal{N}$. Under the same assumptions on the subordinator $S$ as in Theorem \ref{thm:main} we have
$$
\frac{(w(n))^{\alpha+1}K_{n,\,w(n)}}{n^{\alpha}\ell(n/w(n))}\toprobab \alpha \mathcal{I}_{\alpha},
$$
with $\mathcal{I}_{\alpha}$ given by \eqref{eq:exp_func_def}.
\end{theorem}

From the previous two results it is natural to expect that if $q\in\mathcal{Q}\cap\mathcal{N}$ is such that $r(t)=o(q(t))$, then
\begin{equation}\label{eq:to_zero}
K_{n,\,q(n)}\toprobab 0,
\end{equation}
and this is indeed the case, as we shall show later in the proofs. However, it is possible to formulate a non-trivial limit theorem for such ``rapidly increasing'' sequence by considering a closely related functional. For $i\in\N$, put
$$
K_{n,\,\geq i}:=\sum_{k\geq 1}\1_{\{\mathcal{Z}_{n,\,k}\geq i\}},
$$
so $K_{n,\,\geq i}$ is the number of boxes occupied by {\it at least} $i$ balls. Our last main result provides a limit theorem for $K_{n,\,\geq q(n)}$ for arbitrary $q\in\mathcal{Q}\cap\mathcal{N}$ not necessarily satisfying $r(t)=o(q(t))$ nor regularly varying. In a sense the next result is the easiest one, because the sequence $(K_{n,\,\geq i})_{n\in\N}$ is monotone for every fixed $i\in\N$. This also partially explains why almost no assumptions on $q$ are needed.

\begin{theorem}\label{thm:main3}
Under the same assumptions on the subordinator $S$ as in Theorem \ref{thm:main} and for $q\in\mathcal{Q}\cap\mathcal{N}$, the following holds:
$$
\frac{(q(n))^{\alpha}K_{n,\,\geq q(n)}}{n^{\alpha}\ell(n/q(n))}\toprobab \mathcal{I}_{\alpha}.
$$
\end{theorem}

We close the introduction by specializing our main results to stable subordinators, which are typical representatives of the family of subordinators with regularly varying infinite L\'{e}vy measures.

\begin{example}
Let $S$ be an $\alpha$-stable subordinator, that is,
$$
\nu([y,\,\infty))=\frac{y^{-\alpha}}{\Gamma(1-\alpha)},\quad y>0,
$$
for some $\alpha\in(0,1)$. Thus, \eqref{eq:reg_var_nu} holds with $\ell(y)\equiv 1/\Gamma(1-\alpha)$, and \eqref{eq:r_def} holds with
$$
r(t)=\left(\frac{\alpha}{\Gamma(1-\alpha)}\right)^{1/(\alpha+1)}t^{\alpha/(\alpha+1)},\quad t>0.
$$
In particular, taking $m=1$ and $u_1=(\Gamma(1-\alpha)/\alpha)^{1/(\alpha+1)}$ in Theorem \ref{thm:main}, we obtain
$$
K_{n,\,\lfloor n^{\alpha/(\alpha+1)}\rfloor}\todistr \mathcal{P},
$$
where $\mathcal{P}$ has a mixed Poisson distribution with (conditional) mean $\E [\mathcal{P}|S]=\frac{\alpha}{\Gamma(1-\alpha)}\mathcal{I}_{\alpha}$. If $w=w(t)$ is a positive integer-valued function which is regularly varying at infinity and such that $w(t)\to\infty$ and $w(t)=o(t^{\alpha/(\alpha+1)})$, as $t\to\infty$, then
$$
\frac{(w(n))^{\alpha+1}K_{n,\,w(n)}}{n^{\alpha}}\toprobab \frac{\alpha}{\Gamma(1-\alpha)} \mathcal{I}_{\alpha}.
$$
Finally, for arbitrary $q\in\mathcal{Q}\cap\mathcal{N}$, by Theorem \ref{thm:main3}
$$
\frac{(q(n))^{\alpha}K_{n,\,\geq q(n)}}{n^{\alpha}}\toprobab \frac{1}{\Gamma(1-\alpha)} \mathcal{I}_{\alpha}.
$$
\end{example}

\textcolor{black}{The remainder of the paper is organized as follows. In Section \ref{sec:karlin} we discuss our findings from the viewpoint of certain infinite urn models, also known as Karlin occupancy schemes.  Section \ref{sec:karlin} culminates with Theorem \ref{thm:main_ks_poisson}, which is a general Poisson limit theorem for Karlin occupancy schemes. In Section \ref{sec:analytic} we formulate and prove our main technical lemmas on regularly varying functions. The proofs of the main results are given in Sections \ref{sec:proof1}, \ref{sec:proof2} and \ref{sec:proof3}.
}

\section{Poisson limits in the Karlin occupancy scheme}\label{sec:karlin}

\textcolor{black}{The construction of regenerative compositions, as described in the introduction, can be thought of as a particular instance of the following infinite occupancy scheme. Let $(p^{\ast}_k)_{k\in\N}$ be a fixed probability distribution on $\N$ such that $p^{\ast}_k>0$ for all $k\in\N$ and $\sum_{k\geq 1}p^{\ast}_k=1$. Identical balls are allocated one by one in an independent fashion among countable array of boxes such that the probability of hitting box $k$ for a particular ball is equal to $p^{\ast}_k$. This model is known in the literature as {\it the Karlin infinite occupancy scheme} due to seminal paper \cite{Karlin:1967} where its first systematic study was carried out. An excellent survey of the results up to 2007 can be found in \cite{Gnedin+Hansen+Pitman:2007}.}

\textcolor{black}{In essence, the model of regenerative compositions is the Karlin infinite occupancy scheme in random environment, that is, with random frequencies $(p_k)_{k\in\N}$ given by
\begin{equation}\label{eq:rc_frequencies}
p_k=\P\{E\in (S(\tau_k-),\,S(\tau_k))|S\}=e^{-S(\tau_k-)}-e^{-S(\tau_k)},\quad k\in\N,
\end{equation}
where $\{\tau_k\}$ is the set of jump epochs of the subordinator $S$. Throughout the paper we adopt the following convention: various quantities related to a generic Karlin infinite occupancy scheme are starred, whereas the quantities corresponding to the particular scheme with (random) frequencies \eqref{eq:rc_frequencies} are not. Keeping in mind this convention, given that $n$ balls are allocated, we denote by $\mathcal{Z}^{\ast}_{n,k}$ the number of balls in the box $k$, by $K_{n,i}^{\ast}$ the number of boxes occupied by exactly $i$ balls and by $K_{n,\geq i}^{\ast}$ the number of boxes occupied by at least $i$ balls in the Karlin scheme with generic frequencies $(p^{\ast}_k)_{k\in\N}$.}

\textcolor{black}{A key role in the asymptotic analysis of Karlin schemes is played by a counting function
\begin{equation}\label{eq:tho_ast_def}
\rho^{\ast}(x):=\sum_{k\geq 1}\1_{\{p_k^{\ast}\geq x\}},\quad x>0.
\end{equation}
The structure of the occupancy counts vector $(\mathcal{Z}^{\ast}_{n,k})_{k\in\N}$, for large $n$, is regulated to a large extent  by the behavior of this function at zero. A typical assumption is that of 
regular variation, that is, for some $\alpha^{\ast}\in [0,\,1]$, and a function $\ell^{\ast}$ slowly varying at infinity,
\begin{equation}\label{eq:reg_var_ks}
\rho^{\ast}(x)~\sim~x^{-\alpha^{\ast}}\ell^{\ast}(1/x),\quad x\downarrow 0.
\end{equation}
For example, it is known that under this assumption with $\alpha^{\ast}\in (0,\,1]$ the variables $K^{\ast}_{n,r}$, $r\in\N$, and $K_n^{\ast}:=\sum_{r\geq 1}K^{\ast}_{n,r}$ are all of the same magnitude in a sense that there exists a regularly varying normalization $a^{\ast}_n$ such that
$$
\frac{K^{\ast}_{n}}{a^{\ast}_n}\toas 1\quad\text{and}\quad \frac{K^{\ast}_{n,r}}{a^{\ast}_n}\toas c_r,\quad r\in\N,
$$
with the same constants $(c_r)_{r\in\N}$ as in \eqref{eq:k_n_r_reg_var}, see Eq. (20), (21) and Theorem 9 in \cite{Karlin:1967}. In fact, the proof of \eqref{eq:k_n_r_reg_var} in \cite{Gnedin+Pitman+Yor:2006} boils down to showing that \eqref{eq:reg_var_nu}, for $\alpha\in (0,\,1]$, implies
\begin{equation}\label{eq:gne_pit_yor_conv}
\frac{\rho(x)}{x^{-\alpha}\ell(1/x)}\overset{a.s.}{\longrightarrow} \mathcal{I}_{\alpha},\quad x\downarrow 0,
\end{equation}
where
\begin{equation}\label{eq:rho_def}
\rho(x):=\sum_{k\geq 1}\1_{\{p_k\geq x\}}=\sum_{k\geq 1}\1_{\{e^{-S(\tau_k-)}-e^{-S(\tau_k)}\geq x\}},\quad x>0,
\end{equation}
is a random counting function completely determined by the subordinator $S$. See Theorem 5.1 in \cite{Gnedin+Pitman+Yor:2006} for the proof of \eqref{eq:gne_pit_yor_conv} under assumption \eqref{eq:reg_var_nu}. Thus, a question posed in the introduction can also be asked for a general Karlin scheme with frequencies satisfying \eqref{eq:reg_var_ks}: what is a threshold $r^{\ast}=r^{\ast}(n)$ such that $K^{\ast}_{n,r^{\ast}(n)}$ converges to a non-degenerate limit? We shall partially answer this question by proving a Poisson limit theorem for a modified Karlin scheme known as the {\it Poissonized} scheme.}

\textcolor{black}{A powerful tool in the analysis of the Karlin occupancy scheme is the Poissonization--dePoisson\-iza\-tion technique. Let $\Pi=(\Pi(t))_{t\geq 0}$ be a standard Poisson process. The Poissonized model is obtained by allocating the balls at the epochs of the Poisson process $\Pi$: the first ball is dropped at the moment of the first jump of $\Pi$ (which is standard  exponential), the second ball is dropped at the moment of the second jump of $\Pi$ (which is the sum of two independent standard exponentials) and so on. The Poissonized model has an advantage that the number of balls in different boxes form mutually independent homogeneous Poisson processes with intensity of the number of balls in a box $k$ being equal to $p^{\ast}_k$, $k\in\N$. Note that the total number of balls allocated during $[0,\,t]$ is equal to $\Pi(t)$. For $r\in\N$ and $t\geq 0$, denote by $K^{\ast}_r(t):=K^{\ast}_{\Pi(t),\,r}$ the number of boxes containing $r$ balls at time $t$ in the Poissonized model. Note that, for $r\in\N$,
\begin{equation}\label{eq:eq_k_r_t_rep_ks}
K^{\ast}_r(t)=\sum_{k\geq 1}\1_{\{\mathcal{Z}^{\ast}_{\Pi(t),\,k}=r\}},\quad t\geq 0,
\end{equation}
and also
\begin{equation}\label{eq:eq_k_r_t_geq_rep_ks}
K^{\ast}_{\geq r}(t)=\sum_{k\geq 1}\1_{\{\mathcal{Z}^{\ast}_{\Pi(t),\,k}\geq r\}},\quad t\geq 0.
\end{equation}
If the vector $(p^{\ast}_k)_{k\in\N}$ is random (for example, as in \eqref{eq:rc_frequencies}), then the Poisson process $\Pi$ should be taken independent of $(p^{\ast}_k)_{k\in\N}$. In this case $(\mathcal{Z}^{\ast}_{\Pi(t),\,k})_{k\geq 1}$ are {\it conditionally independent} given $(p^{\ast}_k)_{k\in\N}$.}

\textcolor{black}{Our strategy in derivation of Theorem \ref{thm:main} is as follows. We shall first formulate and prove a counterpart of Theorem \ref{thm:main} for a general Karlin occupancy scheme with frequencies satisfying \eqref{eq:reg_var_ks} and $\alpha^{\ast}\in(0,\,1]$ in the Poissonized model\footnote{Note that we do not exclude the case $\alpha^{\ast}=1$.}. By specializing this result to $(p_k)_{k\in\N}$, given by \eqref{eq:rc_frequencies}, we shall then derive the Poissonized versions of Theorem \ref{thm:main} as a corollary. On the third and final step we shall deduce Theorem \ref{thm:main} by applying the dePoisson\-iza\-tion technique. Here is the counterpart of Theorem \ref{thm:main} for the general Karlin scheme satisfying \eqref{eq:reg_var_ks} in the Poissonized settings. The proof will be given in Section \ref{sec:proof1}.}

\textcolor{black}{\begin{theorem}\label{thm:main_ks_poisson}
Assume that \eqref{eq:reg_var_ks} holds for some $\alpha^{\ast}\in (0,\,1]$ and a function $\ell^{\ast}$ slowly varying at infinity. Let $r^{\ast}=r^{\ast}(t)$ by any positive function such that
\begin{equation}\label{eq:tilde_r_def}
\lim_{t\to\infty}\frac{\alpha^{\ast} t^{\alpha^{\ast}}}{(r^{\ast}(t))^{\alpha^{\ast}+1}}\ell^{\ast}\left(\frac{t}{r^{\ast}(t)}\right)=1,
\end{equation}
and $r^{\ast}_i=r^{\ast}_i(t)$, $i=1,\ldots,m$ be integer-valued functions such that $r^{\ast}_i(t)~\sim~u^{\ast}_i r^{\ast}(t)$ for some $0<u^{\ast}_1<u^{\ast}_2<\cdots<u^{\ast}_m<\infty$, as $t\to\infty$. Then
\begin{equation}\label{eq:conv_poisson_ks}
\left(K^{\ast}_{r^{\ast}_1(t)}(t),K^{\ast}_{r^{\ast}_2(t)}(t),\ldots,K^{\ast}_{r^{\ast}_m(t)}(t)\right) \todistrt (\mathcal{P}^{\ast}_1,\mathcal{P}^{\ast}_2,\ldots,\mathcal{P}^{\ast}_m),
\end{equation}
where $(\mathcal{P}^{\ast}_i)_{i=1,\ldots,m}$ are mutually independent Poisson random variables with $\E \left[\mathcal{P}^{\ast}_i \right] = (u^{\ast}_i)^{-\alpha^{\ast}-1}$.
\end{theorem}}

\textcolor{black}{\begin{remark}
Under additional regularity assumptions on the sequence $(p^{\ast}_k)_{k\in\N}$ it is possible to perform dePossonization in Theorem \ref{thm:main_ks_poisson}. For example, if $p^{\ast}_1\geq p^{\ast}_2\geq p^{\ast}_3\geq\cdots$ and
$$
p^{\ast}_j\sim j^{-\beta^{\ast}}\ell_2^{\ast}(j),\quad j\to\infty,
$$
for some $\beta^{\ast}\geq 1$ and a slowly varying function $\ell_2^{\ast}$, then \eqref{eq:reg_var_ks} holds with
$\alpha^{\ast}=1/\beta^{\ast}$ and some slowly varying at infinity function $\ell^{\ast}$, which can be expressed using de Bruijn conjugates as has been explained in Remark \ref{rem:rer_var_r} above. Thus, all the assumptions of 
Theorem \ref{thm:main_ks_poisson} hold. Furthermore, by Proposition 2.2. in \cite{Barbour+Gnedin:2009} applied with $k(n)=\log n$, we also have 
\begin{equation}\label{eq:conv_ks}
\left(K^{\ast}_{n,\,r_1(n)},K^{\ast}_{n,\,r_2(n)},\ldots,K^{\ast}_{n,\,r_m(n)}\right) \todistr (\mathcal{P}^{\ast}_1,\mathcal{P}^{\ast}_2,\ldots,\mathcal{P}^{\ast}_m).
\end{equation}
Unfortunately, such regularity conditions are unavailable  for regenerative compositions, forcing us to perform dePoissonization by hands.
\end{remark}}

\textcolor{black}{The strategy mentioned just before Theorem \ref{thm:main_ks_poisson} can also be applied to derive Theorems \ref{thm:main2} and \ref{thm:main3}. However, their counterparts for the general Karlin occupancy scheme satisfying \eqref{eq:reg_var_ks} are less interesting because the limits are then degenerate (non-random). Instead, we shall derive below the Poissonized versions of Theorems \ref{thm:main2} and \ref{thm:main3} directly, see Propositions \ref{prop:slow_q} and \ref{prop:k_n_r_geq}, and then perform dePossonization.}

\section{Preparatory analytic results}\label{sec:analytic}

The following two analytic lemmas lie in the core of our proofs. Both results provide a kind of Abelian theorems for integrals involving regularly varying functions with the index of regular variation diverging to infinity, and might be of interest on their own.  The essence of the proofs is an application of the saddle-point method.

\begin{lemma}\label{lem:karamata}
Let $\ell$ be a locally bounded positive function slowly varying at infinity and $q\in\mathcal{Q}$. Then
\begin{equation}\label{eq:karamata}
\int_0^{\infty}y^{q(t)-1}e^{-y}\ell(t/y){\rm d}y~\sim~\Gamma(q(t))\ell(t/q(t)),\quad t\to\infty.
\end{equation}
\end{lemma}
Note that for $q(t)\equiv q>0$ this is just the direct half of the classical Karamata theorem. Also note that in general $\ell(t/q(t))$ in the right-hand side cannot be replaced by $\ell(t)$.
\begin{proof}
Let $a,A$ be positive constants such that $0<a<1<A<\infty$. Write
$$
\int_0^{\infty}y^{q(t)-1}e^{-y}\ell(t/y){\rm d}y=\left(\int_0^{aq(t)}+\int_{aq(t)}^{Aq(t)}+\int_{Aq(t)}^{\infty}\right)y^{q(t)-1}e^{-y}\ell(t/y){\rm d}y\\
=:I_1(t)+I_2(t)+I_3(t).
$$
We start by showing that
\begin{equation}\label{eq:karamata_proof1}
\lim_{t\to\infty}\frac{I_1(t)}{\Gamma(q(t))\ell(t/q(t))}=0.
\end{equation}
For $\alpha>0$, the function $y\mapsto y^{\alpha}e^{-y}$ is increasing on $[0,\,\alpha]$. Since $aq(t)\leq q(t)-1$ for large enough $t>0$, we obtain
\begin{equation}\label{eq:karamata_proof1bis}
\int_0^{aq(t)}y^{q(t)-1}e^{-y}\ell(t/y){\rm d}y\leq (aq(t))^{q(t)-1}e^{-aq(t)}\int_0^{aq(t)}\ell(t/y){\rm d}y
\sim~(aq(t))^{q(t)}e^{-aq(t)}\ell(t/(aq(t))),\quad t\to\infty,
\end{equation}
where the last passage follows from Proposition 1.5.10 in \cite{BGT} upon substitution $z=t/y$. Thus, by the slow variation of $\ell$ and the Stirling formula for the gamma-function,
$$
\lim_{t\to\infty}\frac{I_1(t)}{\Gamma(q(t))\ell(t/q(t))}\leq \lim_{t\to\infty}\frac{(aq(t))^{q(t)}e^{-aq(t)}}{\Gamma(q(t))}=\lim_{t\to\infty}\frac{(aq(t))^{q(t)}e^{-aq(t)}}{\sqrt{2\pi}q(t)^{q(t)-1/2}e^{-q(t)}}=\lim_{t\to\infty}\sqrt{q(t)/(2\pi)}(ae^{-a}e)^{q(t)}.
$$
The last limit is equal to zero because $ae^{-a}e<1$ for $0<a<1$ and \eqref{eq:karamata_proof1} follows. Let us check that
\begin{equation}\label{eq:karamata_proof2}
\lim_{t\to\infty}\frac{I_3(t)}{\Gamma(q(t))\ell(t/q(t))}=0.
\end{equation}
Upon substitution $y=q(t)z$ we obtain
\begin{multline*}
I_3(t)=(q(t))^{q(t)}\int_{A}^{\infty}z^{q(t)-1}e^{-q(t)z}\ell(t/(zq(t))){\rm d}z\leq t^{-1}(q(t))^{q(t)+1}\sup_{y\leq t/(Aq(t))}(y\ell(y))\int_{A}^{\infty}z^{q(t)}e^{-q(t)z}{\rm d}z.
\end{multline*}
By Theorem 1.5.3 in \cite{BGT}, the right-hand side of the last display is asymptotically equal to
$$
A^{-1}(q(t))^{q(t)}\ell(t/q(t))\int_{A}^{\infty}z^{q(t)}e^{-q(t)z}{\rm d}z.
$$
Therefore, \eqref{eq:karamata_proof2} is a consequence of
\begin{equation}\label{eq:karamata_proof3}
\lim_{t\to\infty}\frac{(q(t))^{q(t)}\int_{A}^{\infty}z^{q(t)}e^{-q(t)z}{\rm d}z}{\Gamma(q(t))}=0.
\end{equation}
To check the latter we write
\begin{align}
\int_{A}^{\infty}z^{q(t)}e^{-q(t)z}{\rm d}z&=\int_0^{\infty}(z+A)^{q(t)}e^{-q(t)(A+z)}{\rm d}z\notag\\
&=(Ae^{-A})^{q(t)}\int_0^{\infty}(1+z/A)^{q(t)}e^{-zq(t)}{\rm d}z\notag\\
&\leq (Ae^{-A})^{q(t)}\int_0^{\infty}e^{-z(1-A^{-1})q(t)}{\rm d}z\notag\\
&=(Ae^{-A})^{q(t)}\frac{A}{(A-1)q(t)}.\label{eq:karamata_proof1bisbis}
\end{align}
Thus, the Stirling formula for the gamma-function yields \eqref{eq:karamata_proof3} because $Ae^{-A}e<1$ for $A>1$. By applying \eqref{eq:karamata_proof1} and \eqref{eq:karamata_proof2} with $\ell\equiv 1$ we obtain
$$
\int_{aq(t)}^{Aq(t)}y^{q(t)-1}e^{-y}{\rm d}y~\sim~\Gamma(q(t)),\quad t\to\infty.
$$
This shows that \eqref{eq:karamata} is a consequence of
$$
\lim_{t\to\infty}\frac{I_2(t)}{\ell(t/q(t))\int_{aq(t)}^{Aq(t)}y^{q(t)-1}e^{-y}{\rm d}y}=1.
$$
But this relation follows trivially from the uniform convergence theorem for slowly varying functions, see Theorem 1.5.2 in \cite{BGT}. Indeed,
$$
\frac{\inf_{z\in [a,\,A]}\ell(t/(zq(t)))}{\ell(t/q(t))}\leq\frac{I_2(t)}{\ell(t/q(t))\int_{aq(t)}^{Aq(t)}y^{q(t)-1}e^{-y}{\rm d}y}\leq\frac{\sup_{z\in [a,\,A]}\ell(t/(zq(t)))}
{\ell(t/q(t))},
$$
and both lower and upper bounds converge to one, as $t\to\infty$, by the aforementioned uniformity.
\end{proof}
\begin{corollary}\label{cor:karamata}
Let $U$ be a positive measurable and locally bounded function such that $U(x)~\sim~x^{\beta}\ell(x)$, as $x\to\infty$, for some $\ell$ slowly varying at infinity and $\beta\in\R$. If $q\in\mathcal{Q}$, then
\begin{equation}\label{eq:karamata_reg_var}
\int_0^{\infty}y^{q(t)-1}e^{-y}U(t/y){\rm d}y~\sim~\Gamma(q(t))U(t/q(t)),\quad t\to\infty.
\end{equation}
\end{corollary}
\begin{proof}
Without loss of generality we may assume that $U(x)=x^{\beta}\ell(x)$ for all $x>0$. The proof now follows follows from Lemma \ref{lem:karamata} by plugging $U(t/y)=(t/y)^{\beta}\ell(t/y)$ and using the asymptotic relation $\Gamma(q(t)-\beta)~\sim~(q(t))^{-\beta}\Gamma(q(t))$, since $q(t)\to\infty$, as $t\to\infty$.
\end{proof}

The next lemma is a counterpart of the Abelian implication of the Karamata theorem for Laplace--Stieltjes transforms.

\begin{lemma}\label{lem:karamata_ls}
Let $q\in\mathcal{Q}$ and $U:(0,\,\infty)\mapsto (0,\,\infty)$ be a right-continuous nonincreasing function such that $U(x)~\sim~x^{-\gamma}\ell(1/x)$, as $x\downarrow 0$, for some $\ell$ slowly varying at infinity and $\gamma>0$. Then
\begin{equation}\label{eq:karamata_ls}
\int_{[0,\,\infty)}e^{-tx}\frac{(tx)^{q(t)}}{\Gamma(q(t)+1)}{\rm d}(-U(x))~\sim~\gamma U(q(t)/t)/q(t)~=~\frac{\gamma t^{\gamma}\ell(t/q(t))}{q^{1+\gamma}(t)},\quad t\to\infty.
\end{equation}
\end{lemma}
\begin{proof}
The equality in \eqref{eq:karamata_ls} is obvious. Making change of variable $x=q(t)y/t$ and integrating by parts, we obtain
\begin{align*}
\int_{[0,\,\infty)}e^{-tx}(tx)^{q(t)}{\rm d}(-U(x))&=(q(t))^{q(t)}\int_{[0,\,\infty)}(ye^{-y})^{q(t)}{\rm d}_y(-U(q(t)y/t))\\
&=(q(t))^{q(t)}\int_{[0,\,\infty)}U(q(t)y/t){\rm d}_y\left((ye^{-y})^{q(t)}\right)\\
&=(q(t))^{q(t)}\left(\int_{[0,\,y_1)}+\int_{[y_1,\,y_2]}+\int_{(y_2,\,\infty)}\right)\cdots\\
&=(q(t))^{q(t)}(J_1(y_1,t)+J_2(y_1,y_2,t)+J_3(y_2,t)),
\end{align*}
where $y_1<1<y_2$ are the real roots of the equation $ye^{-y}=1/4$, see Fig.~\ref{graph}. Arguing exactly as in the proof of equations \eqref{eq:karamata_proof1} and \eqref{eq:karamata_proof2}, it can be checked the integrals $J_1(y_1,t)$ and $J_3(y_2,t)$ are negligible, that is,
\begin{equation}\label{eq:karamata_ls_proof11}
\int_{[0,\,\infty)}U(q(t)y/t){\rm d}_y\left((ye^{-y})^{q(t)}\right)~\sim~\int_{[y_1,\,y_2]}U(q(t)y/t){\rm d}_y\left((ye^{-y})^{q(t)}\right),\quad t\to\infty.
\end{equation}
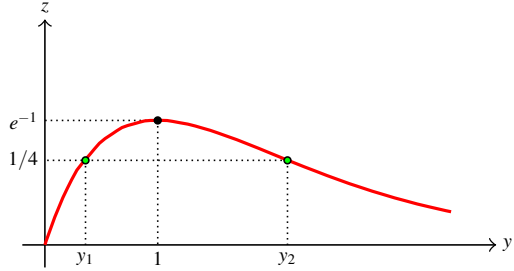
\begin{figure}[!ht]
\centering
 \scalebox{1.5}{\begin{tikzpicture}
\draw[->] (-0.2,0) -- (4.0,0) node[right,scale=0.5] {$y$};
\draw[->] (0,-0.2) -- (0,2) node[above,scale=0.5] {$z$};
\draw [densely dotted] (1.0,3*0.368) -- (1.0,0) node[below,scale=0.5] {$1$};
\draw [densely dotted] (1.0,3*0.368) -- (0.0,3*0.368) node[left,scale=0.5] {$e^{-1}$};
\draw [densely dotted] (2.15,3*0.25) -- (0.0,3*0.25) node[left,scale=0.5] {$1/4$};
\draw [densely dotted] (0.36,3*0.25) -- (0.36,3*0.0) node[below,scale=0.5] {$y_1$};
\draw [densely dotted] (2.15,3*0.25) -- (2.15,3*0.0) node[below,scale=0.5] {$y_2$};
\draw [thick,red] plot [smooth, tension=1] coordinates {(0,3*0) (0.2,3*0.168) (0.4,3*0.268) (0.6, 3*0.329) (0.8, 3*0.359) (1.0, 3*0.368) (1.2,3*0.361) (1.4, 3*0.345) (1.6,3*0.323) (1.8,3*0.298) (2.0,3*0.271) (2.2, 3*0.244) (2.4, 3*0.218) (2.6, 3*0.193) (2.8,3*0.17) (3.0,3*0.149)
(3.2,3*0.13) (3.4,3*0.113) (3.6,3*0.098)};
\draw[fill=black] (1.0,3*0.368) circle (0.03);
\draw[fill=green] (0.36,3*0.25) circle (0.03);
\draw[fill=green] (2.15,3*0.25) circle (0.03);
\end{tikzpicture}}
\caption{The plot of the function $y\mapsto ye^{-y}$ (solid red). The scale on the $z$-axis is three times larger than on the $y$-axis. The solid black dot is the maximum $e^{-1}$ attained at $y=1$, the green dots are the intersections with the horizontal line $z=1/4$. Two inverse functions $W_1$ and $W_2$ are obtained by inverting monotone pieces of $y\mapsto ye^{-y}$ on $[y_1,\,1]$ and $[1,\,y_2]$, respectively.}
\label{graph}
\end{figure}

Let us check that
\begin{equation}\label{eq:karamata_ls1}
\frac{(q(t))^{q(t)}}{\Gamma(1+q(t))}\int_{[y_1,\,y_2]}U(q(t)y/t){\rm d}_y\left((ye^{-y})^{q(t)}\right)~\sim~\gamma U(q(t)/t)/q(t),\quad t\to\infty.
\end{equation}
To this end, let $W_1:[1/4,\,e^{-1}]\mapsto [y_1,\,1]$ and $W_2:[1/4,\,e^{-1}]\mapsto [1,\,y_2]$ be the inverses of the function $y\mapsto ye^{-y}$ on its monotonicity intervals $[y_1,\,1]$ and $[1,\,y_2]$, respectively, see Fig.~\ref{graph}. Then, using the changes of variables $y=W_{1,2}(z^{1/q(t)})$, we obtain
\begin{align*}
&\hspace{-1cm}\int_{[y_1,\,y_2]}U\left(\frac{q(t)y}{t}\right){\rm d}_y\left((ye^{-y})^{q(t)}\right)\\
&=\left(\int_{[y_1,\,1]}+\int_{[1,\,y_2]}\right)U\left(\frac{q(t)y}{t}\right){\rm d}_y\left((ye^{-y})^{q(t)}\right)\\
&=\int_{[4^{-q(t)},\,e^{-q(t)}]}\left(U\left(\frac{q(t)}{t}W_1(z^{1/q(t)})\right)-U\left(\frac{q(t)}{t}W_2(z^{1/q(t)})\right)\right){\rm d}z.
\end{align*}
In view of the inequalities
$$
y_1\leq W_{1}(z^{1/q(t)})\leq 1\quad\text{and}\quad 1\leq W_{2}(z^{1/q(t)})\leq y_2,
$$
which hold for all $z\in [4^{-q(t)},\,e^{-q(t)}]$, the uniform convergence theorem for regularly varying functions, see Theorem 1.5.2 in \cite{BGT}, is applicable. Thus, for every fixed $\varepsilon\in (0,1)$ there exists $t_0>0$ such that for all $t>t_0$ and all $z\in [4^{-q(t)},\,e^{-q(t)}]$, it holds
\begin{multline*}
(1-\varepsilon)U(q(t)/t)\left((W_1(z^{1/q(t)}))^{-\gamma}-(W_2(z^{1/q(t)}))^{-\gamma}\right)\leq\\
U\left(\frac{q(t)}{t}W_1(z^{1/q(t)})\right)-U\left(\frac{q(t)}{t}W_2(z^{1/q(t)})\right)\leq\\
(1+\varepsilon)U(q(t)/t)\left((W_1(z^{1/q(t)}))^{-\gamma}-(W_2(z^{1/q(t)}))^{-\gamma}\right).
\end{multline*}
These inequalities demonstrate that \eqref{eq:karamata_ls1} follows from
\begin{equation}\label{eq:karamata_ls_proof12}
\lim_{t\to\infty}\frac{(q(t))^{q(t)+1}}{\Gamma(1+q(t))}\int_{[4^{-q(t)},\,e^{-q(t)}]}\left((W_1(z^{1/q(t)}))^{-\gamma}-(W_2(z^{1/q(t)}))^{-\gamma}\right){\rm d}z~=~\gamma.
\end{equation}
This latter can be checked by reversing the arguments. The changes of variables $y=W_{1,2}(z^{1/q(t)})$ yield
\begin{align*}
&\hspace{-1cm}\frac{(q(t))^{q(t)+1}}{\Gamma(1+q(t))}\int_{[4^{-q(t)},\,e^{-q(t)}]}\left((W_1(z^{1/q(t)}))^{-\gamma}-(W_2(z^{1/q(t)}))^{-\gamma}\right){\rm d}z\\
&=\frac{(q(t))^{q(t)+1}}{\Gamma(1+q(t))}\int_{[y_1,\,y_2]}y^{-\gamma}{\rm d}_y\left((ye^{-y})^{q(t)}\right)\\
&=\frac{(q(t))^{q(t)+1}}{\Gamma(1+q(t))}\frac{(q(t))^{\gamma}}{t^\gamma}\int_{[y_1,\,y_2]}(q(t)y/t)^{-\gamma}{\rm d}_y\left((ye^{-y})^{q(t)}\right)\\
&\sim~\frac{(q(t))^{q(t)+1}}{\Gamma(1+q(t))}\frac{(q(t))^{\gamma}}{t^\gamma}\int_{[0,\,\infty)}(q(t)y/t)^{-\gamma}{\rm d}_y\left((ye^{-y})^{q(t)}\right),\quad t\to\infty,
\end{align*}
where the last passage follows from \eqref{eq:karamata_ls_proof11} applied with $U(y)=y^{-\gamma}$. It remains to note that
$$
\int_{[0,\,\infty)}y^{-\gamma}{\rm d}_y\left((ye^{-y})^{q(t)}\right)=\gamma\int_0^{\infty}y^{q(t)-\gamma-1}e^{-yq(t)}{\rm d}y=\gamma (q(t))^{\gamma-q(t)}\Gamma(q(t)-\gamma),
$$
which readily implies \eqref{eq:karamata_ls_proof12}. The proof is complete.
\end{proof}

\begin{remark}\label{lem:karamata_ls_increasing}
Using the same method we can also prove the analogue of Lemma \ref{lem:karamata_ls} for nondecreasing integrators $U$. More precisely, if $q\in\mathcal{Q}$ and $U:(0,\,\infty)\mapsto (0,\,\infty)$ is a right-continuous nondecreasing function such that $U(x)=x^{\gamma}\ell(1/x)$ for some $\ell$ slowly varying at infinity and $\gamma>0$, then 
\begin{equation}\label{eq:karamata_ls_increasing}
\int_{[0,\,\infty)}e^{-tx}\frac{(tx)^{q(t)}}{\Gamma(q(t)+1)}{\rm d}U(x)~\sim~\gamma U(q(t)/t)/q(t)~=~\frac{\gamma t^{-\gamma}\ell(t/q(t))}{q^{1-\gamma}(t)},\quad t\to\infty.
\end{equation}
\end{remark}

\section{Proof of Theorem \ref{thm:main_ks_poisson}}\label{sec:proof1}
\textcolor{black}{It is enough to show that, for arbitrary fixed $|z_i|\leq 1$, $i=1,\ldots,m$,
\begin{equation}\label{eq:gen_func_conv_cond_ks}
\E\left[\prod_{i=1}^{m}z_i^{K^{\ast}_{r^{\ast}_i(t)}(t)}\right]\to\exp\left(-\sum_{i=1}^{m}(u^{\ast}_i)^{-\alpha^{\ast}-1}(1-z_i)\right),\quad t\to\infty.
\end{equation}
This yields \eqref{eq:conv_poisson_ks} because the right-hand side is a multivariate generating function of $(\mathcal{P}^{\ast}_1,\mathcal{P}^{\ast}_2,\ldots,\mathcal{P}^{\ast}_m)$.}

\textcolor{black}{To prove \eqref{eq:gen_func_conv_cond_ks} pick $t_0>0$ so large that $r^{\ast}_i(t) < r^{\ast}_j(t)$ for all $t>t_0$ and $1\leq i<j\leq m$. By the independence of $(\mathcal{Z}^{\ast}_{\Pi(t),\,k})_{k\geq 1}$ we have
$$
\E\left[\prod_{i=1}^{m}z_i^{K^{\ast}_{r^{\ast}_i(t)}(t)}\right]=\prod_{k\geq 1}\E \left[\prod_{i=1}^{m}z_i^{\1_{\{\mathcal{Z}^{\ast}_{\Pi(t),\,k}=r^{\ast}_i(t)\}}}\right]\\
=\prod_{k\geq 1}\left(1-\sum_{i=1}^{m}(1-z_i)\P\{\mathcal{Z}^{\ast}_{\Pi(t),\,k}=r^{\ast}_i(t)\}\right).
$$
By an elementary fact, see Lemma 4.8 in \cite{Kallenberg}, the right-hand side of the last display converges to the right-hand side of \eqref{eq:gen_func_conv_cond_ks} provided
\begin{equation}\label{prop:des1_ks}
\sum_{k\geq 1}\P\{\mathcal{Z}^{\ast}_{\Pi(t),\,k}=r^{\ast}_i(t)\}\to (u^{\ast}_i)^{-\alpha^{\ast}-1},\quad t\to\infty,
\end{equation}
for all $i=1,\ldots,m$. Using that the distribution of $\mathcal{Z}^{\ast}_{\Pi(t),\,k}$ is Poisson with parameter $p_k^{\ast}t$ and recalling the definition \eqref{eq:tho_ast_def} of the function $\rho^{\ast}$, we can write
$$
\sum_{k\geq 1}\P\{\mathcal{Z}^{\ast}_{\Pi(t),\,k}=r^{\ast}_i(t)\}=\int_{(0,\,\infty)}e^{-tx}\frac{(tx)^{r^{\ast}_i(t)}}{r^{\ast}_i(t)!}{\rm d}(-\rho^{\ast}(x)),\quad i=1,\ldots,m.
$$
Note that $\rho^{\ast}$ is nonincreasing and $\rho^{\ast}(x)=0$ for $x\geq 1$. In what follows we shall frequently write integrals with $-\rho^{\ast}(x)$ (or $-\rho(x)$) in the integrator over infinite intervals $(0,\,\infty)$ keeping in mind that that the actual domain of integration is $(0,1)$.}

\textcolor{black}{Applying Lemma \ref{lem:karamata_ls} with $U(x)=\rho^{\ast}(x)$ and $q(t)=r^{\ast}_i(t)$ we infer 
\begin{equation}\label{eq:k_n_r_mean_ks}
\sum_{k\geq 1}\P\{\mathcal{Z}^{\ast}_{\Pi(t),\,k}=r^{\ast}_i(t)\}~\sim~\frac{\alpha^{\ast} t^{\alpha^{\ast}}}{(r^{\ast}_i(t))^{\alpha^{\ast}+1}}\ell^{\ast}(t/r^{\ast}_i(t)),\quad t\to\infty,
\end{equation}
for all $i=1,\ldots,m$. Since, by \eqref{eq:tilde_r_def},
$$
\frac{\alpha^{\ast} t^{\alpha^{\ast}}\ell^{\ast}(t/r^{\ast}_i(t))}{(r^{\ast}_i(t))^{\alpha^{\ast}+1}}~\sim~\frac{\alpha^{\ast} t^{\alpha^{\ast}}\ell^{\ast}(t/r^{\ast}(t))}{(u^{\ast}_i r^{\ast}(t))^{\alpha^{\ast}+1}}~\to~(u^{\ast}_i)^{-\alpha^{\ast}-1},\quad t\to\infty,
$$
we obtain \eqref{prop:des1_ks} for all $i=1,\ldots,m$. The proof is complete.}

\section{Proofs of Theorems \ref{thm:main}, \ref{thm:main2} and \ref{thm:main3}: the Poissonized model}\label{sec:proof2}
\textcolor{black}{Recalling our convention on using the ``starred'' and ``unstarred'' notation, put
\begin{equation}\label{eq:eq_k_r_t_rep}
K_r(t)=\sum_{k\geq 1}\1_{\{\mathcal{Z}_{\Pi(t),\,k}=r\}},\quad t\geq 0,
\end{equation}
and also
\begin{equation}\label{eq:eq_k_r_t_geq_rep}
K_{\geq r}(t)=\sum_{k\geq 1}\1_{\{\mathcal{Z}_{\Pi(t),\,k}\geq r\}},\quad t\geq 0.
\end{equation}
Here is the Poissonized version of Theorem \ref{thm:main}. Note that the case $\alpha=1$ is included.
\begin{proposition}\label{prop:main_poisson}
Assume that \eqref{eq:reg_var_nu} holds for some $\alpha\in (0,\,1]$ and let $r$ and $r_i$ be as in Theorem \ref{thm:main}.
$$
\left(K_{r_1(t)}(t),K_{r_2(t)}(t),\ldots,K_{r_m(t)}(t)\right) \todistrt (\mathcal{P}_1,\mathcal{P}_2,\ldots,\mathcal{P}_m).
$$
\end{proposition}
\begin{proof}
According to formula \eqref{eq:gne_pit_yor_conv} and in view of assumption \eqref{eq:reg_var_nu} we know that for a.s.~every sample path of the subordinator $S$, it holds
$$
\rho(x)~\sim~x^{-\alpha}\ell(1/x)\mathcal{I}_{\alpha},\quad x\downarrow 0.
$$
Thus, we can apply Theorem \ref{thm:main_ks_poisson} pathwise with 
\begin{multline*}
\alpha^{\ast}=\alpha,\quad \ell^{\ast}(x)=\mathcal{I}_{\alpha}\ell(x),\quad
r^{\ast}(t)=(\mathcal{I}_{\alpha})^{1/(1+\alpha)}r(t),\\
r^{\ast}_i(t)=r_i(t)\quad\text{and}\quad u^{\ast}_i=(\mathcal{I}_{\alpha})^{-1/(1+\alpha)}u_i\quad i=1,\ldots,m,
\end{multline*}
to conclude that the conditional distribution of $\left(K_{r_1(t)}(t),K_{r_2(t)}(t),\ldots,K_{r_m(t)}(t)\right)$, given $S$, converges weakly, as $t\to\infty$, to the conditional distribution of $(\mathcal{P}_1,\mathcal{P}_2,\ldots,\mathcal{P}_m)$ for a.s.~every sample path of $S$, that is,
$$
\P\left\{\left(K_{r_1(t)}(t),K_{r_2(t)}(t),\ldots,K_{r_m(t)}(t)\right)\in \cdot \Big| S\right\}\toast \P\left\{\left((\mathcal{P}_1,\mathcal{P}_2,\ldots,\mathcal{P}_m)\right)\in \cdot \Big| S\right\},\quad t\to\infty.
$$
By the Lebesgue dominated convergence theorem this convergence holds also unconditionally, thereby finishing the proof of Theorem \ref{prop:main_poisson}.
\end{proof}}

The next proposition is the counterpart of Theorem \ref{thm:main2}. Note that we do not assume regular variation of $w$ here.
\begin{proposition}\label{prop:slow_q}
Assume that $w\in\mathcal{W}^{r}\cap\mathcal{N}$. Then \eqref{eq:reg_var_nu} implies
$$
\frac{(w(t))^{\alpha+1}K_{w(t)}(t)}{t^{\alpha}\ell(t/w(t))}\toprobabt \alpha \mathcal{I}_{\alpha}.
$$
\end{proposition}
\begin{proof}
From representation \eqref{eq:eq_k_r_t_rep} and Lemma \ref{lem:karamata_ls} applied with $U(x)=\rho(x)$ and $q(t)=w(t)$ it follows that
\begin{equation}\label{eq:cond_mean_q}
\E [K_{w(t)}(t)|S]=\sum_{k\geq 1}\P\{\mathcal{Z}_{\Pi(t),\,k}=w(t)|S\}~\sim~\frac{\alpha\mathcal{I}_{\alpha}t^{\alpha}}{(w(t))^{\alpha+1}}\ell(t/w(t)),\quad t\to\infty\quad\text{a.s.},
\end{equation}
and by our assumptions on $w$ the right-hand side is divergent to $+\infty$ a.s. Further, from \eqref{eq:eq_k_r_t_rep} it follows that
$$
{\rm Var}\, [K_{w(t)}(t)|S]=\sum_{k\geq 1}{\rm Var}\, [\1_{\{\mathcal{Z}_{\Pi(t),\,k}=w(t)\}}|S] \leq \E [K_{w(t)}(t)|S].
$$
Thus, by Chebyshev's inequality, for every fixed $\varepsilon>0$,
$$
\P\left\{\left|\frac{K_{w(t)}(t)}{\E [K_{w(t)}(t)|S]}-1\right|>\varepsilon \Big| S \right\}\toast 0.
$$
By the dominated convergence theorem
$$
\frac{K_{w(t)}(t)}{\E [K_{w(t)}(t)|S]}\toprobabt 1,
$$
which, by Slutsky's lemma and  \eqref{eq:cond_mean_q}, yields the desired claim.
\end{proof}
Finally, here is a counterpart of Theorem \ref{thm:main3} for the Poissonized model.
\begin{proposition}\label{prop:k_n_r_geq}
Assume that $q\in\mathcal{Q}\cap\mathcal{N}$. Then \eqref{eq:reg_var_nu} implies
$$
\frac{(q(t))^{\alpha}K_{\geq q(t)}(t)}{t^{\alpha}\ell(t/q(t))}\toprobabt \mathcal{I}_{\alpha}.
$$
\end{proposition}
\begin{proof}
Representation \eqref{eq:eq_k_r_t_geq_rep} implies
\begin{align*}
\E [K_{\geq q(t)}(t)|S]&=\sum_{k\geq 1}\P\{\mathcal{Z}_{\Pi(t),\,k}\geq q(t)|S\}=\int_{(0,\,\infty)}\left(\sum_{j\geq q(t)}e^{-tx}\frac{(tx)^j}{j!}\right){\rm d}(-\rho(x))\\
&=t\int_{(0,\,\infty)}\rho(x)e^{-tx}\frac{(tx)^{q(t)-1}}{(q(t)-1)!}{\rm d}x=\int_{(0,\,\infty)}\rho(y/t)e^{-y}\frac{y^{q(t)-1}}{(q(t)-1)!}{\rm d}x,
\end{align*}
where the penultimate equality follows upon integration by parts from an easily checked relation
$$
\frac{{\rm d}}{{\rm d}\lambda}\left(\sum_{j=k}^{\infty}e^{-\lambda}\frac{\lambda^j}{j!}\right)=e^{-\lambda}\frac{\lambda^{k-1}}{(k-1)!},\quad k\in\N.
$$
Thus, Corollary \ref{cor:karamata} applied with $U(x)=\rho(1/x)$ yields
$$
\E [K_{\geq q(t)}(t)|S]~\sim~\rho(q(t)/t)~\sim~\mathcal{I}_{\alpha}(t/q(t))^{\alpha}\ell(t/q(t)),\quad t\to\infty\quad\text{a.s.},
$$
by \eqref{eq:gne_pit_yor_conv}. Since $t/q(t)\to\infty$, as $t\to\infty$, we have $\E [K_{\geq q(t)}(t)|S]\to\infty$ a.s. From the inequality ${\rm Var}\, [K_{\geq q(t)}(t)|S]\leq \E [K_{\geq q(t)}(t)|S]$, we obtain
$$
\frac{K_{\geq q(t)}(t)}{\E [K_{\geq q(t)}(t)|S]}\toprobabt 1,
$$
by Chebyshev's inequality and dominated convergence. This completes the proof of Proposition \ref{prop:k_n_r_geq}.
\end{proof}

Arguing in the same vein as in the derivation of \eqref{eq:cond_mean_q} it can be easily checked that whenever $q\in\mathcal{Q}\cap\mathcal{N}$ is such that $r(t)=o(q(t))$, then
$$
\P\{K_{q(t)}(t)\neq 0|S\}=\P\{K_{q(t)}(t)\geq 1|S\}\leq \E[K_{q(t)}(t)|S]\toast 0.
$$
Thus,
\begin{equation}\label{eq:to_zero_poisson}
K_{q(t)}(t)\toprobabt 0,
\end{equation}
which is the analogue of \eqref{eq:to_zero} for the Poissonized model.

\section{Proofs of Theorems \ref{thm:main}, \ref{thm:main2} and \ref{thm:main3}: dePoissonization}\label{sec:proof3} This part of the proofs is called dePoissonization and its aim is to deduce Theorems \ref{thm:main}, \ref{thm:main2} and \ref{thm:main3} from Propositions \ref{prop:main_poisson}, \ref{prop:slow_q} and \ref{prop:k_n_r_geq}, respectively. \textcolor{black}{There exists a number of general results of this flavour in the literature. For example, the aforementioned Proposition 2.2 in \cite{Barbour+Gnedin:2009}. However, we have not been able to apply them without imposing extra regularity assumptions. Our approach works only under the assumption $\alpha<1$ explaining the appearance of this condition in our main results.}

We start with an easier implication Proposition \ref{prop:k_n_r_geq} $\,\,\Rightarrow\,\,$ Theorem \ref{thm:main3}. It is simpler because $(K_{n,\,\geq i})_{n\in\N}$ is nondecreasing.

\subsection{Proof of Theorem \ref{thm:main3} using Proposition \ref{prop:k_n_r_geq}.} Fix $\varepsilon\in(0,\,1)$ and note that on the event
$\{\Pi((1-\varepsilon)t)\leq \lfloor t\rfloor \leq \Pi((1+\varepsilon)t)  \}$ we have
$$
\frac{(q(\lfloor t\rfloor))^{\alpha}K_{\geq q(\lfloor t\rfloor)}((1-\varepsilon)t)}{\lfloor t\rfloor^{\alpha}\ell(\lfloor t\rfloor/q(\lfloor t\rfloor))}
\leq\frac{(q(\lfloor t\rfloor))^{\alpha}K_{\lfloor t\rfloor,\,\geq q(\lfloor t\rfloor)}}{\lfloor t\rfloor^{\alpha}\ell(\lfloor t\rfloor/q(\lfloor t\rfloor))}\leq \frac{(q(\lfloor t\rfloor))^{\alpha}K_{\geq q(\lfloor t\rfloor)}((1+\varepsilon)t)}{\lfloor t\rfloor^{\alpha}\ell(\lfloor t\rfloor/q(\lfloor t\rfloor))}.
$$
Put $q^{\pm\varepsilon}(t):=q(\lfloor t/(1\pm\varepsilon) \rfloor)$ and note that $q\in\mathcal{Q}\cap\mathcal{N}$ implies $q^{\pm\varepsilon}\in\mathcal{Q}\cap\mathcal{N}$. Further,
\begin{multline*}
\frac{(q(\lfloor t\rfloor))^{\alpha}K_{\geq q(\lfloor t\rfloor)}((1\pm \varepsilon)t)}{\lfloor t\rfloor^{\alpha}\ell(\lfloor t\rfloor/q(\lfloor t\rfloor))}=\frac{(q^{\pm\varepsilon}(t(1\pm\varepsilon)))^{\alpha}K_{q^{\pm\varepsilon}(t(1\pm\varepsilon))}((1\pm\varepsilon)t)}{\lfloor t\rfloor^{\alpha}\ell(\lfloor t\rfloor/q^{\pm\varepsilon}(\lfloor (1\pm\varepsilon)t\rfloor))}\\
\sim~\frac{(q^{\pm\varepsilon}(t(1\pm\varepsilon)))^{\alpha}K_{q^{\pm\varepsilon}(t(1\pm\varepsilon))}((1\pm\varepsilon)t)}{\lfloor t\rfloor^{\alpha}\ell(\lfloor (1\pm\varepsilon)t\rfloor/q^{\pm\varepsilon}(\lfloor (1\pm\varepsilon)t\rfloor))},\quad t\to\infty\quad\text{a.s.}
\end{multline*}
by the slow variation of $\ell$. By Proposition \ref{prop:k_n_r_geq} the right-hand side converges in probability to $(1\pm\varepsilon)^{\alpha}\mathcal{I}_{\alpha}$, as $t\to\infty$. Sending $\varepsilon\downarrow 0$ and noting that
\begin{equation}\label{eq:wlln_poisson}
\lim_{t\to\infty}\P\{\Pi((1-\varepsilon)t)\leq \lfloor t\rfloor \leq \Pi((1+\varepsilon)t)\}=1,
\end{equation}
we arrive at the conclusion of Theorem \ref{thm:main3}.

For the remaining implications, namely Proposition \ref{prop:main_poisson} $\,\,\Rightarrow\,\,$ Theorem \ref{thm:main} and Proposition \ref{prop:slow_q} $\,\,\Rightarrow\,\,$ Theorem \ref{thm:main2}, more sophisticated arguments are necessary due to lack of monotonicity of $(K_{n,\,i})_{n\in\N}$, $i\in\N$. In particular, the first order result for $(\Pi(t))$ given by \eqref{eq:wlln_poisson} is not sufficient.

\subsection{Proofs of Theorems \ref{thm:main} and \ref{thm:main2} using Propositions \ref{prop:main_poisson} and \ref{prop:slow_q}.}
We shall prove both implications simultaneously. Throughout this subsection $v(t)$ denotes either $r_i(t)$, for some fixed $i=1,\ldots,m$, in the settings of Theorem \ref{thm:main} or $w(t)$ in the settings of Theorem \ref{thm:main2}. In both cases $v$ is regularly varying and integer-valued. Set
$$
c(t)~:=~t^{\alpha}\ell(t/v(\lfloor t\rfloor))/(v(\lfloor t\rfloor))^{\alpha+1}~\sim~t^{\alpha}\ell(t/v(t))/(v(t))^{\alpha+1},\quad t\to\infty.
$$
It is enough to prove
\begin{equation}\label{eq:depoisson1}
\frac{K_{v(\lfloor t\rfloor)}(t)-K_{\lfloor t\rfloor,\,v(\lfloor t\rfloor)}}{c(t)}=\frac{K_{\Pi(t),\,v(\lfloor t\rfloor)}-K_{\lfloor t\rfloor,\,v(\lfloor t\rfloor)}}{c(t)}\toprobabt 0.
\end{equation}
Indeed, under the assumptions of Theorem \ref{thm:main}, $c(t)\to c$, as $t\to\infty$, for some $c>0$. Therefore, \eqref{eq:depoisson1} is equivalent to $\P\{K_{v(\lfloor t\rfloor)}(t)\neq K_{\lfloor t\rfloor,\,v(\lfloor t\rfloor)}\}\to 0$, as $t\to\infty$. This proves that Proposition \ref{prop:main_poisson} and \eqref{eq:depoisson1} yield Theorem \ref{thm:main}. Similarly, Proposition \ref{prop:slow_q} and \eqref{eq:depoisson1} yield Theorem \ref{thm:main2}.

Fix $T>0$, $\delta>0$ and note that for large enough $t>0$ we have
\begin{multline*}
\P\{|K_{\Pi(t),\,v(\lfloor t\rfloor )}-K_{\lfloor t\rfloor,\,v(\lfloor t\rfloor)}|\geq \delta c(t)\}\leq
1-\P\{\Pi(t-T\sqrt{t})\leq \lfloor t\rfloor \leq \Pi(t+T\sqrt{t}) \}\\
+\P\{\Pi(t-T\sqrt{t})\leq \lfloor t\rfloor \leq \Pi(t+T\sqrt{t}) ,\,|K_{\Pi(t),\,v(\lfloor t\rfloor )}-K_{\lfloor t\rfloor,\,v(\lfloor t\rfloor)}|\geq \delta c(t)\}.
\end{multline*}
By the central limit theorem for Poisson processes we see that it is enough to check
$$
\lim_{T\to\infty}\limsup_{t\to\infty}\P\{\Pi(t-T\sqrt{t})\leq \lfloor t\rfloor \leq \Pi(t+T\sqrt{t}),|K_{\Pi(t),\,v(\lfloor t\rfloor )}-K_{\lfloor t\rfloor,\,v(\lfloor t\rfloor)}|\geq \delta c(t)\}=0,
$$
for every $\delta>0$, which in turn follows from
\begin{equation}\label{eq:depoisson_proof2}
\limsup_{t\to\infty}\frac{\E\left(\sup_{s\in [t-T\sqrt{t},\,t+T\sqrt{t}]}|K_{\Pi(s),\,v(\lfloor t\rfloor )}-K_{\Pi(t),\,v(\lfloor t\rfloor)}|\right)}{c(t)}=0,
\end{equation}
for arbitrary fixed $T>0$. In order to prove \eqref{eq:depoisson_proof2} we estimate the supremum as follows:
$$\sup_{s\in [t-T\sqrt{t},\,t+T\sqrt{t}]}|K_{\Pi(s),\,v(\lfloor t\rfloor )}-K_{\Pi(t),\,v(\lfloor t\rfloor)}|\leq \sum_{k\geq 1}\sup_{s\in [t-T\sqrt{t},\,t+T\sqrt{t}]}|\1_{\{\mathcal{Z}_{\Pi(s),\,k}=v(\lfloor t\rfloor )\}}-\1_{\{\mathcal{Z}_{\Pi(t),\,k}=v(\lfloor t\rfloor )\}}|.
$$

Note that the $k$-th summand on the right-hand side can be non-zero only in the following two scenarios:
\begin{itemize}
\item the number of balls in the $k$-th box at time $t-T\sqrt{t}$ was strictly smaller than $v(\lfloor t\rfloor)$ but at time $t+T\sqrt{t}$ it became larger or equal than $v(\lfloor t\rfloor)$, that is, at {\it some} epoch during $[t-T\sqrt{t},\,t+T\sqrt{t}]$ the number of balls in the $k$-th box increased to $v(\lfloor t\rfloor)$;
\item at least one ball has fallen during $[t-T\sqrt{t},\,t+T\sqrt{t}]$ in the $k$-th box which contained exactly $v(\lfloor t\rfloor)$ balls at time $t-T\sqrt{t}$.
\end{itemize}
Thus,
\begin{multline*}
\sup_{s\in [t-T\sqrt{t},\,t+T\sqrt{t}]}|K_{\Pi(s),\,v(\lfloor t\rfloor )}-K_{\Pi(t),\,v(\lfloor t\rfloor)}|\leq\\
\sum_{k\geq 1}\left(\left(\sum_{j=0}^{v(\lfloor t\rfloor)-1}\1_{\{\mathcal{Z}_{\Pi(t-T\sqrt{t}),\,k}=j,\mathcal{Z}_{\Pi(t+T\sqrt{t}),\,k}\geq v(\lfloor t\rfloor )\}}\right)+\1_{\{\mathcal{Z}_{\Pi(t-T\sqrt{t}),\,k}=v(\lfloor t\rfloor ),\mathcal{Z}_{\Pi(t+T\sqrt{t}),\,k}> v(\lfloor t\rfloor )\}}\right),
\end{multline*}
and thereupon
\begin{align*}
&\E\left(\sup_{s\in [t-T\sqrt{t},\,t+T\sqrt{t}]}|K_{\Pi(s),\,v(\lfloor t\rfloor )}-K_{\Pi(t),\,v(\lfloor t\rfloor)}|\right)\leq\\
&\hspace{0.6cm}\leq\sum_{k\geq 1} \sum_{j=0}^{v(\lfloor t\rfloor)-1}\P\{\mathcal{Z}_{\Pi(t-T\sqrt{t}),\,k}=j,\mathcal{Z}_{\Pi(t+T\sqrt{t}),\,k}\geq v(\lfloor t\rfloor)\}\\
&\hspace{0.6cm}+ \sum_{k\geq 1}\P\{\mathcal{Z}_{\Pi(t-T\sqrt{t}),\,k}=v(\lfloor t\rfloor),\mathcal{Z}_{\Pi(t+T\sqrt{t}),\,k} > v(\lfloor t\rfloor)\}\\
&\hspace{0.6cm}=\sum_{k\geq 1}\P\{v(\lfloor t\rfloor)-(\mathcal{Z}_{\Pi(t+T\sqrt{t}),\,k}-\mathcal{Z}_{\Pi(t-T\sqrt{t}),\,k})\leq \mathcal{Z}_{\Pi(t-T\sqrt{t}),\,k}\leq v(\lfloor t\rfloor)-1\}\\
&\hspace{0.6cm}+\sum_{k\geq 1}\P\{\mathcal{Z}_{\Pi(t-T\sqrt{t}),\,k}=v(\lfloor t\rfloor),\mathcal{Z}_{\Pi(t+T\sqrt{t}),\,k}-\mathcal{Z}_{\Pi(t-T\sqrt{t}),\,k} \geq 1\}.
\end{align*}
Recall that, conditional on $S$, $(\mathcal{Z}_{\Pi(s),\,k})_{s\geq 0}$ is a Poisson process and in particular has independent and stationary increments. Thus, exploiting definition \eqref{eq:rho_def} of the function $\rho$ and Markov's inequality for the second summand, we obtain
\begin{align*}
&\E\left(\sup_{s\in [t-T\sqrt{t},\,t+T\sqrt{t}]}|K_{\Pi(s),\,v(\lfloor t\rfloor )}-K_{\Pi(t-T\sqrt{t}),\,v(\lfloor t\rfloor)}|\right)\\
&\hspace{0.3cm}\leq \E\sum_{k\geq 1}\P\{v(\lfloor t\rfloor)-(\mathcal{Z}_{\Pi(t+T\sqrt{t}),\,k}-\mathcal{Z}_{\Pi(t-T\sqrt{t}),\,k})\leq \mathcal{Z}_{\Pi(t-T\sqrt{t}),\,k}\leq v(\lfloor t\rfloor)-1|S\}\\
&\hspace{0.3cm}+\E\sum_{k\geq 1}\P\{\mathcal{Z}_{\Pi(t-T\sqrt{t}),\,k}=v(\lfloor t\rfloor)|S\}\P\{\mathcal{Z}_{\Pi(t+T\sqrt{t}),\,k}-\mathcal{Z}_{\Pi(t-T\sqrt{t}),\,k} \geq 1|S\}\\
&\hspace{0.3cm}= \int_{(0,\,\infty)}\P\{v(\lfloor t\rfloor)-{\rm Poi}^{\prime}(2T\sqrt{t}x) \leq {\rm Poi}((t-T\sqrt{t})x)\leq  v(\lfloor t\rfloor)-1\}{\rm d}(-\E\rho(x))\\
&\hspace{0.3cm}+\int_{(0,\,\infty)}e^{-x(t-T\sqrt{t})}\frac{x^{v(\lfloor t\rfloor)}(t-T\sqrt{t})^{v(\lfloor t\rfloor)}}{v(\lfloor t\rfloor)!}(2Tx\sqrt{t}){\rm d}(-\E\rho(x))\\
&\hspace{0.3cm}=:P_1(t)+P_2(t),
\end{align*}
where ${\rm Poi}$ and ${\rm Poi}^{\prime}$ are independent Poisson random variables. We first deal with $P_2(t)$. By Theorem 5.1 in \cite{Gnedin+Pitman+Yor:2006}
\begin{equation}\label{eq:exp_rho_reg_var}
\E\rho(x)~\sim~{\rm const}\cdot x^{-\alpha}\ell(1/x),\quad x\to\infty,
\end{equation}
where here and below ``${\rm const}$'' denotes some positive constants which does not depend on $x$ and/or $t$ but might depend on all other parameters. Let $(s(t))_{t\geq 0}$ be such that $s(t)-T\sqrt{s(t)}=t$ for all $t\geq 0$. Lemma \ref{lem:karamata_ls} applied with $U(x)=\E\rho(x)$, $t$ replaced by $s(t)$ and $q(t)=v(\lfloor s(t)\rfloor)+1$, yields
\begin{multline*}
P_2(s(t))=\frac{2T\sqrt{s(t)}(v(\lfloor s(t)\rfloor)+1)}{t}\int_{(0,\,\infty)}e^{-xt}\frac{(tx)^{v(\lfloor s(t)\rfloor)+1}}{(v(\lfloor s(t)\rfloor)+1)!}{\rm d}\left(-\E\rho (x)\right)\\
\sim~{\rm const}\cdot\frac{2T\sqrt{s(t)}(v(\lfloor s(t)\rfloor)+1)}{t}\alpha t^{\alpha}\ell(t/(v(\lfloor s(t)\rfloor)+1))/(v(\lfloor s(t)\rfloor)+1)^{1+\alpha},\quad t\to\infty.
\end{multline*}
Since $v$ is regularly varying both in the settings of Theorem \ref{thm:main} and Theorem \ref{thm:main2}, the relation $s(t)~\sim~t$, as $t\to\infty$, implies $v(\lfloor s(t)\rfloor)+1 ~\sim~v(t)$ and thereupon
$$
P_2(s(t))~\sim~{\rm const}\cdot \frac{t^{\alpha-1/2}\ell(t/v(t))}{(v(t))^{\alpha}},\quad t\to\infty.
$$
If $v=r_i$, then the index of regular variation of $v$ is $\alpha/(\alpha+1)<1/2$, see Remark \ref{rem:rer_var_r} and thereupon
\begin{equation}\label{eq:depoisson_proof3}
\lim_{t\to\infty}\frac{t^{\alpha-1/2}\ell(t/r_i(t))}{r_i^{\alpha}(t)}=0=\lim_{t\to\infty}\frac{t^{\alpha-1/2}\ell(t/r_i(t))}{r_i^{\alpha}(t)c(t)}.
\end{equation}
If $v=w\in\mathcal{W}^{r}_{RV}$, that is, we are in settings of Theorem \ref{thm:main2}, then
\begin{equation}\label{eq:depoisson_proof4}
\lim_{t\to\infty}\frac{t^{\alpha-1/2}\ell(t/w(t))}{w^{\alpha}(t)c(t)}=\lim_{t\to\infty}\frac{w(t)}{t^{1/2}}=0.
\end{equation}
holds as well, since we assume $w(t)=o(r(t))$ and $r(t)=o(\sqrt{t})$, as $t\to\infty$.

In order to estimate $P_1(t)$ we decompose it as follows: for fixed $0<a<1<A$,
$$
P_1(t)=\left(\int_{(0,\,av(\lfloor t\rfloor)/t)}+\int_{[av(\lfloor t\rfloor)/t,\,Av(\lfloor t\rfloor)/t]}+\int_{(Av(\lfloor t\rfloor)/t,\,\infty)}\right)\cdots~=:~P_{11}(t)+P_{12}(t)+P_{13}(t).
$$
As far as $P_{12}(t)$ is concerned  we have, by the stochastic monotonicity of ${\rm Poi}(\lambda)$ in parameter $\lambda$,
$$
P_{12}(t)\leq \int_{[av(\lfloor t\rfloor)/t,\,Av(\lfloor t\rfloor)/t]}\P\{v(\lfloor t\rfloor)-{\rm Poi}^{\prime}(2ATv(\lfloor t\rfloor)/\sqrt{t})\leq{\rm Poi}((t-T\sqrt{t})x)\leq  v(\lfloor t\rfloor)-1\}{\rm d}(-\E\rho(x)).
$$
Pick $M\in\N$ so large that $\alpha<M/(M+2)$ and note that
\begin{align*}
P_{12}(t) &\leq \P\{{\rm Poi}^{\prime}(2ATv(\lfloor t\rfloor)/\sqrt{t})\geq M\}\int_{[av(\lfloor t\rfloor)/t,\,Av(\lfloor t\rfloor)/t]}{\rm d}(-\E\rho(x))\\
&+\sum_{j=1}^{M-1}\P\{{\rm Poi}^{\prime}(2ATv(\lfloor t\rfloor)/\sqrt{t})=j\}\sum_{i=1}^{j}\int_{[av(\lfloor t\rfloor)/t,\,Av(\lfloor t\rfloor)/t]}\P\{{\rm Poi}((t-T\sqrt{t})x)=v(\lfloor t\rfloor)-i\}{\rm d}(-\E\rho(x))\\
&\leq {\rm const}\cdot \left(\frac{v(t)}{\sqrt{t}}\right)^{M}\E \rho(av(\lfloor t\rfloor)/t)\\
&+{\rm const}\cdot\sum_{j=1}^{M-1}\left(\frac{v(t)}{\sqrt{t}}\right)^{j}\sum_{i=1}^{j}\int_{(0,\,\infty)}\P\{{\rm Poi}((t-T\sqrt{t})x)=v(\lfloor t\rfloor)-i\}{\rm d}(-\E\rho(x)),
\end{align*}
where the bound $\P\{{\rm Poi}(\lambda)\geq M\}\leq \lambda^M$ has been utilized for the estimate of $\P\{{\rm Poi}^{\prime}(2ATv(\lfloor t\rfloor)/\sqrt{t})\geq M\}$. The  first term goes to zero, as $t\to\infty$, by \eqref{eq:exp_rho_reg_var} and the choice of $M$. By Lemma \ref{lem:karamata_ls} every integral in the second term is $O(t^{\alpha}\ell(t/v(t))/(v(t))^{\alpha+1})$ (with a possibly dependent on $j=1,\ldots,M$ constant in the Landau symbol). Thus, by \eqref{eq:depoisson_proof3} and \eqref{eq:depoisson_proof4} all summands in the second term tend to zero upon division by $c(t)$, as $t\to\infty$.

For the estimates of $P_{11}(t)$ and $P_{13}(t)$ we employ known bounds for Poisson tail probabilities borrowed from \cite{Glynn:1987}.
By part (ii) of Proposition 1 in \cite{Glynn:1987}, we have, for large enough $t>0$,
\begin{align*}
P_{11}(t)&\leq \int_{(0,\,av(\lfloor t\rfloor)/t)}\P\{v(\lfloor t\rfloor)\leq {\rm Poi}((t+T\sqrt{t})x)\}{\rm d}(-\E\rho(x)),\\
& \leq\int_{(0,\,av(\lfloor t\rfloor)/t)}\left(1-\frac{(t+T\sqrt{t})x}{1+v(\lfloor t\rfloor)}\right)^{-1}\P\{{\rm Poi}((t+T\sqrt{t})x)=v(\lfloor t\rfloor)\}{\rm d}(-\E\rho(x))\\
&\leq {\rm const}\cdot\int_{(0,\,av(\lfloor t\rfloor)/t)}\P\{{\rm Poi}((t+T\sqrt{t})x)=v(\lfloor t\rfloor)\}{\rm d}(-\E\rho(x)),
\end{align*}
and this converges to zero exponentially fast, which is readily seen upon integration by parts, the bound
$$
\left|\frac{{\rm d}\P\{{\rm Poi}(\lambda)=j\}}{{\rm d}\lambda}\right|\leq \P\{{\rm Poi}(\lambda)=j\}+\P\{{\rm Poi}(\lambda)=j-1\},\quad j\in\N,
$$
and \eqref{eq:karamata_proof1bis}. Similarly, by Proposition 1(i) in \cite{Glynn:1987}
\begin{align*}
P_{13}(t)&\leq \int_{(Av(\lfloor t\rfloor)/t,\,\infty)}\P\{{\rm Poi}((t-T\sqrt{t})x)\leq v(\lfloor t\rfloor)\}{\rm d}(-\E\rho(x))\\
& \leq\int_{(Av(\lfloor t\rfloor)/t,\,\infty)}\left(1-\frac{v(\lfloor t\rfloor)}{(t+T\sqrt{t})x}\right)^{-1}\P\{{\rm Poi}((t+T\sqrt{t})x)=v(\lfloor t\rfloor)\}{\rm d}(-\E\rho(x))\\
&\leq {\rm const}\cdot\int_{(Av(\lfloor t\rfloor)/t,\,\infty)}\P\{{\rm Poi}((t+T\sqrt{t})x)=v(\lfloor t\rfloor)\}{\rm d}(-\E\rho(x)),
\end{align*}
and this also converges to zero in view of \eqref{eq:karamata_proof1bisbis}. This finishes the proof of \eqref{eq:depoisson_proof2} and of Theorem \ref{thm:main}.

Using similar arguments it can be checked that \eqref{eq:to_zero_poisson} implies \eqref{eq:to_zero}.

\section{Acknowledgments}

The work of Alexander Marynych has received funding from the Ulam Program of the Polish National Agency for Academic Exchange (NAWA), project no. PPN/ULM/2019/1/00004/DEC/1. D.~Buraczewski was  partially supported by the National Science Center, Poland (grant number  2019/33/ B/ST1/00207). The authors would like to thank Zakhar Kabluchko for many helpful discussions. \textcolor{black}{We are most grateful to two anonymous referees for many valuable comments which have led to significant improvements of the presentation of our work.}

\end{document}